\documentclass[reqno, 12pt]{amsart}
\usepackage[mathscr]{eucal}
\usepackage{amscd}
\usepackage{color}
\usepackage{amsfonts}
\usepackage{amsmath, amsthm, amssymb}
\usepackage{amsmath,amssymb,amsfonts,amsthm,bbm,mathrsfs,verbatim}
\usepackage[misc]{ifsym}
\usepackage{fancyhdr}
\usepackage{xr-hyper}
\usepackage{latexsym}
\usepackage{psfrag}
\usepackage{enumitem}

\usepackage{graphicx}

\numberwithin{equation}{section}
\newcommand{\R}{\mathbb{R}}

\newcommand{\si}{\sigma}

\newcommand{\Om}{\Omega}
\newcommand{\ol}{\overline}

\newcommand{\no}{\nonumber}

\newcommand{\vep}{\varepsilon}
\newcommand{\intom}{\int_{\Omega}}
\newcommand{\avint}{\rlap{$-$}\!\int_{\Omega}}
\newcommand{\avintr}{\rlap{$-$}\!\int_{B_1(0)}}
\newcommand{\prt}{\partial}

\newcommand{\la}{\lambda}
\newcommand{\dx}{\mathrm{d}x}

\newcommand{\de}{\delta}

\newtheorem{theorem}{Theorem}[section]
\newtheorem{lemma}{Lemma}[section]
\newtheorem{proposition}{Proposition}[section]
\newtheorem{corollary}{Corollary}[section]
\newtheorem{rem}{Remark}[section]
\newtheorem{definition}{Definition}[section]

\def\be{\begin{equation}}
\def\ee{\end{equation}}
\def\bge{\begin{eqnarray}}
\def\bgee{\begin{eqnarray*}}
\def\ege{\end{eqnarray}}
\def\egee{\end{eqnarray*}}
\newcommand{\bea}{\begin{eqnarray}}
\newcommand{\eea}{\end{eqnarray}}
\newcommand{\bean}{\begin{eqnarray*}}
\newcommand{\eean}{\end{eqnarray*}}
\begin{document}

\title{Diffusion-driven blow-up for a non-local Fisher-KPP type
model}

\author{Nikos I. Kavallaris}
\address{
 Department of Mathematical and Physical Sciences, University of Chester, Thornton Science Park
Pool Lane, Ince, Chester  CH2 4NU, UK
}

\email{n.kavallaris@chester.ac.uk}

\author{Evangelos A. Latos}

\address{Institute for Mathematics and Scientific Computing, Karl-Franzens-University Graz, Heinr. 36, A-8010 Graz, Austria}

\email{evangelos.latos@uni-graz.at}



\date{\today}
\maketitle
\begin{abstract}
The purpose  of the current paper is to unveil the key mechanism which is responsible for the occurrence of {\it Turing-type instability} for a non-local Fisher-KPP model. In particular, we prove that the solution of the considered non-local Fisher-KPP equation in the neighbourhood of a constant stationary solution, is destabilized  via a {\it diffusion-driven blow-up}. It is also shown that the observed {\it diffusion-driven blow-up} is complete, whilst its blow-up rate is completely classified. Finally, the detected {\it diffusion-driven  instability} results in the formation of unstable blow-up patterns, which are also identified through the determination of the blow-up profile of the solution.



\end{abstract}

\section{Introduction}
\subsection*{The mathematical model}
In as early as 1952, A. Turing in his seminal paper \cite{t52} attempted, by using reaction-diffusion
systems, to model the phenomenon of morphogenesis, the regeneration of tissue structures
in hydra, an animal of a few millimeters in length made up of approximately $100,000$ cells.
Further observations on the morphogenesis in hydra led to the assumption of the existence
of two chemical substances (morphogens), a slowly diffusing (short-range) activator and a
rapidly diffusing (long-range) inhibitor. A. Turing, in \cite{t52}, indicates that although diffusion
has a smoothing and trivializing effect on a single chemical, for the case of the interaction of
two or more chemicals different diffusion rates could force the uniform steady states of the
corresponding reaction–diffusion systems to become unstable and to lead to non-homogeneous
distributions of such reactants. Since then, such a phenomenon is now known as {\it Turing-type instability} or {\it diffusion-driven instability (DDI)};  however such a phenomenon has been first specified in \cite{R40}.

The main purpose of the current paper is the investigation  of the occurrence of  a {\it Turing-type} or {\it DDI instability} for the following non-local Fisher-KPP type model
\begin{align}\label{fkpp}
 &\ u_t-\Delta u=|u|^{p-1}u\left(1-\sigma\avint |u|^{\beta-1}u\;dx\right), &&\hspace{-5em}x\in\Om,\; t>0,\\
  &\ \frac{\partial u}{\partial \nu}=0, &&\hspace{-5em} x\in\prt\Om,\;t>0\label{nbc},
  \\
 &\ u(x,0)=u_0(x)\geq 0, &&\hspace{-5em}x\in\Om.\label{id}
\end{align}
Our motivation to investigate the possible {\it Turing-type} instability of the above model stems from the fact that the non-local Fisher-KPP equation \eqref{fkpp} arises as a mathematical model in several research areas. In particular, \eqref{fkpp} characterizes the evolution of a population of density $u$ when its individuals are moving either by diffusion and/or by  interaction. Actually, the fate  of the population is determined by the interaction modus which might lead either to growth or decay.The reaction term describes the joint influence of a nonlinear growth accounting for a weak Allee effect and of concurrence for available resources (prevention of overcrowding). The nonlocal form of the reaction term infers that several individuals of the population interact in a space/phenotypic trait/etc. domain, through sampling all occupancy information therein. This kind of  problems arise e.g., when modeling emergence and evolution of a biological species cf. \cite{BH94, B00, D04, FG89, LL97, vol2}. Thereby the respective population is structured by a phenotypical trait and its individuals infer two essential interactions: mutation and selection. From this perspective $u(x, t)>0$  serves as the density of a population having phenotype $x$ at time $t.$ The mutation process, is described by a diffusion operator on the trait space,  and it is modeled by a classical diffusion operator, whereas the selection process is illustrated by the nonlocal term $u^p\left(1-\sigma\avint u^{\beta}\;dx\right),$  where $\sigma>0$ stands for the (non-local) parameter measuring the intensity of the selection process.
 Equation \eqref{fkpp} has been  also proposed, cf. \cite{GP07,GVA06}, as a simple model of adaptive dynamics, where again the variable $x$ represents a
phenotypical trait of a given population. The individuals of such a population with trait $x$ face competition
from all their counterparts which does not depend on the trait itself. Other types of non-local terms may arise, see \cite{CD05,PS05} for dispersal by jumps rather than by the Brownian motion. Note also that the imposed Neumann type boundary condition \eqref{nbc} describes the fact that the population does not interact with its external environment.
Besides,  $\Om$ in \eqref{fkpp}  is assumed to be a bounded domain in $\mathbb{R}^N$, $N\geq3$, with boundary of class $C^{2,r}$ for some $r\in(0,1).$ We also consider $u_0\in L^\beta(\Om)\setminus\{0\}$ and the involved exponents $p, \beta$  are set to satisfy
\begin{equation}\label{mt1}
 p\geq\beta>1.
\end{equation}
Equation \eqref{fkpp} is actually a non-local version of the well known Fisher-KPP equation was first introduced, in its scalar form,
\bge\label{lfkpp}
\frac{\partial u}{\partial t} = \frac{\partial^2u}{\partial x^2} + u^p(1-u).
\ege
by Fisher \cite{f37} and Kolmogorov, Petrovskii, Piskunov \cite{kpp}, both in 1937, in the context of population dynamics. Here $u$ represents the population density and the reaction term in \eqref{lfkpp} is considered to be the reproduction rate of the population. When $p = 1$, this reproduction rate is proportional to the population density $u$ and to the available resources $(1-u).$ While, when $p = 2$ the model actually takes into account the addition of sexual reproduction with the reproduction rate to be  proportional to the square of the population density, see \cite{VVpp, volpet, vol1, vol2}. Later, in 1938, Zeldovich and Frank-Kamenetskii \cite{zfk38} came up with equation \eqref{lfkpp} in combustion theory where now $u$ stands for the temperature of the combustive mixture.



In the literature, far more cases of non-local problems are encountered  where the non-local terms induced by an integral of the solution over the domain of interaction $\Omega,$ cf. \cite{AlKH11, KN07, KTz, KLW17,  KS18, L1, L2, LTz1, LTz2, QS, s98, Tz02} and the references there in; however a non-local reaction term close to the one of \eqref{fkpp} is particularly considered in \cite{bds93, hy95,SJM07,SK08}.  Notably, in  \cite{bds93} the authors considered a non-local parabolic reaction-diffusion of the form 
\bge\label{gnlt}
u_t=\Delta u+u^p-\frac{1}{|\Omega|}\int_\Omega u^p\;dx.
\ege
For the case $p=2$ and for $\Omega=(0,1)$ they proved the finite time blow-up of the solutions by considering appropriate initial data. Equation \eqref{gnlt} for general exponent $p$ was also considered in \cite{hy95} and the authors, among others, proved that the solution can blow-up if $p>N/(N-2)$ by considering  spiky initial data. Later on, the authors in \cite{SJM07,SK08} proved the occurrence of finite-time blow-up for \eqref{gnlt} even for $p>1$ and initial data satisfying an energy inequality, utilizing a gamma convergence argument in order to get appropriate lower bounds for the considered Lyapunov functional. Thanks to the negative sign of the non-local reaction term included in \eqref{fkpp} and  \eqref{gnlt} maximum principle fails and thus comparison methods are not applicable, cf. \cite[Proposition 52.24]{QS}. Furthermore, reaction-diffusion equation \eqref{gnlt} leads to a conservation of the total mass, which is a key property for the investigation of its dynamics; it also admits a Lyapunov functional a helpful tool for the derivation of a priori estimates of the solution. In contrast, equation \eqref{fkpp} lacks these two key features although the associated total mass is still  bounded, a crucial property still used for the investigation of its dynamics.

Now regarding the non-local reaction-diffusion equation \eqref{fkpp} there are some already available results in the literature.
More precisely,  the authors in \cite{BChL} proved that the problem \eqref{fkpp}-\eqref{id} for $\beta=1$ admits global-in-time solutions for
$N=1,2$ with any $1 \le p <2$ or $N \ge 3$ when $1 \le
p < 1+2/N.$  Moreover, in \cite{BChL} the asymptotic convergence of solutions towards the
solution of the heat equation is also proved. Some more existence results were shown for the whole space case, i.e. when $\Omega=\R^N$ as well as for different boundary conditions in \cite{B,BCh}; we refer the interested reader to these works for more references about this kind of problems. Finally, in \cite{LLC} the authors considered \eqref{fkpp} on $\mathbb{R}$ and studied the wave fronts of the corresponding nonlinear non-local bistable reaction-diffusion equation. Finally, quite recently,  in \cite{LCS20},  the whole space case with reaction term $u^p\left(1-\sigma J*u^{\beta}\right)$ for a proper kernel $J$  is  investigated. Nevertheless, as far as we know there are no blow-up results available  in the literature for the non-local equation \eqref{fkpp}, and so in the current paper we will try to fill in this gap by providing some blow-up results for the Neumann problem \eqref{fkpp}-\eqref{id}.


\subsection*{Main results}
In the current subsection the main results of our work  related with the occurrence of a a {\it Turing-type} instability for model are demonstrated. First, it is worth noting, that due to the power non-linearity and thanks to condition \eqref{mt1},  if a {\it Turing-type} (or {\it  (DDI)) instability} occurs for the solution of non-local problem \eqref{fkpp}-\eqref{id}, then it should lead to the non existence of global-in-time solutions. More precisely, such an instability would be exhibited in the form of a {\it diffusion-driven blow-up (DDBU)}, cf. \cite{FN05, hy95}.

In this work we restrict ourselves to the radial symmetric case, i.e. when $\Omega = B_1$ where
\[ B_1=B_1(0):=\{x\in \R^N: |x|<1\},\]
denotes the unit sphere in $\R^N.$ Then the solution of \eqref{pfkpp}-\eqref{pid} is radial symmetric, cf. \cite{gnn}, that is $u(x, t) = u(r, t)$ for $0\leq r=|x|<1$ and so problem \eqref{pfkpp}-\eqref{pid} is reduced to
\bea
 && u_t- \Delta_r u
 =
 K(t)u^p,\quad 0<r<1,\; 0<t<T,\label{fkpp2a}\\
  &&  u_r(0,t)=u_r(1,t)=0,\quad 0<t<T,\label{fkpp2b}
  \\
 &&\ u(r,0)=u_0(r)\geq 0, \quad 0<r<1,\label{id2}
\eea
where  $T > 0$ is the maximal existence time, $\Delta_r:=\frac{\partial^2}{\partial r^2}+\frac{N-1}{r} \frac{\partial }{\partial r}$ and
\bean
K(t)\equiv
1
-\sigma \avintr u^\beta\;dx.
\eean
Notably the absolute values have been dropped, since the solution of problem \eqref{fkpp2a}-\eqref{id2} is nonegative when nongeative initial data are considered, cf.  Lemma \ref{pos}.

Next we consider, as in \cite{hy95, KS16, LN09}, spiky initial data of the form
\be\label{sid}
u_0(r)=\lambda\phi_\delta(r),\quad\mbox{for}\quad 0<\lambda<<1,
\ee
where
\begin{align}\label{fd}
\phi_\delta(r)=
\begin{cases}
	r^{-a},\quad  &\delta\leq r\leq1,\\
	\delta^{-a}\left(1+\frac{a}{2}\right)-\frac{a}{2}\delta^{-(a+2)}r^2,\quad  &0\leq r<\delta,
\end{cases}
\end{align}
for $ a :=\frac{2}{p-1}$ and $\delta\in(0,1).$  Taking also into account that  $u_0'(r)<0$, then $\max_{r\in[0,1]}u=u_0(0),$ and via maximum principle for the heat operator, since $K(t)>0$ due to Lemma \ref{Lemma:mbetaestimate} $(i)$,  we also deduce that $u_r(r,t)<0,$ hence $||u(\cdot,t)||_{\infty}=u(0,t).$

Henceforth, we will denote by $T_{\de}$  the maximum existence time of solution of \eqref{fkpp2a}-\eqref{id2} with initial data given by \eqref{fd} and \eqref{sid}.
In the sequel we  prove that this kind of  initial  can lead to  finite-time blow-up for the solution of problem \eqref{fkpp2a}-\eqref{id2}, i.e. to the occurrence  of $T_{\de}<+\infty$
 such that
\be\label{sbu}
||u(\cdot,t)||_{\infty}=u(0,t)\to +\infty \quad\mbox{as}\quad t\to T_{\de}.
\ee
Our first main result  is stated as follows:
\begin{theorem}\label{thmbu}
Let $\Omega=B_1\subset \R^N$ with $N\geq 3,$  $p>\frac{N}{N-2}$ and \eqref{mt1} hold. Then there is a $\lambda_0>0$ provided with the following property: any $0<\lambda\leq \lambda_0$ admits $0<\delta_0=\delta_0(\lambda)<1$ such that any solution of problem \eqref{fkpp2a}-\eqref{id2} with initial data of the form \eqref{sid}-\eqref{fd} satisfying Lemma \ref{Lemma:mbetaestimate} $(i)$ and $0<\delta\leq \delta_0$ blows up in finite time, i.e. $T_{\de}<+\infty.$
\end{theorem}
As far as we are aware Theorerm  \ref{thmbu} is the first available blow-up result   in the literature for non-local problem \eqref{fkpp}-\eqref{id}.
\begin{rem}
Theorem \ref{thmbu} guarantees the occurrence of a diffusion-induced blow-up. Namely it can be easily seen that any spatial homogeneous solution of \eqref{fkpp2a}-\eqref{id2} initiating close to the steady-stade solution $u_{\infty}\equiv \sigma^{-\frac{1}{\beta}},$ and solving  the IVP
\bean
\frac{dU}{dt}=U^{p}\left(1-\sigma U^{\beta}\right),\; t>0,\; U(0)=U_0,
\eean
is stable and it converges to the steady state solution $u_{\infty}.$ Otherwise,  Theorem \ref{thmbu} states that such a solution destabilizes once diffusion enters into the equation.
\end{rem}

It is known, see for example \cite[Proposition 52.24]{QS}, that the maximum principle is not applicable for the non-local problem \eqref{sid}-\eqref{fd} and hence comparison techniques fail. Therefore, our main strategy to overcome this obstacle is to derive a lower estimate of the non-local term $K(t)$
and then deal with a local problem for which comparison techniques are available. Although a lower estimate of $K(t)$ is provided by Lemma \ref{Lemma:mbetaestimate}, such an estimate is not uniform in time and thus an alternative approach should be applied to derive a uniform lower bound. To that end we will follow an approach used in \cite{hy95, KS16, KS18}, and which was actually inspired by ideas in \cite{fmc85}. The steps of the proposed approach, though, needs to be modified appropriately so we can tackle the technical  difficulties arise from the very different non-local term of problem \eqref{fkpp2a}-\eqref{id2} compared with the one considered in problems discussed in  \cite{hy95, KS16, KS18}. It is worth pointing out that the underlying method can be also implemented to predict {\it  diffusion-driven blow-up (DDBU)} even in the case of an isotropically evolving domain $\Omega(t),t>0$, for more details see \cite{KBM}.

Next, the form of the {\it DDBU} provided by Theorem \ref{thmbu} is further investigated. As a complementary result we show, cf. Corollary \ref{nik2}, that as soon as the solution of problem \eqref{fkpp2a}-\eqref{id2} blows up in finite time $T_{\de}<\infty,$ then it immediately becomes unbounded along the whole domain $\Omega$ at any subsequent time; such a phenomenon is known in the literature as {\it complete blow-up}.  In other words, the observed Turing-type instability is quite severe so it destroys all the occurring instability patterns once the blow-up time is exceeded.

Our next main result, identifying the blow-up (Turing-type instability) rate,  is presented below:
\begin{theorem}\label{tbu}
Let $N\geq 3$ with  $p>\frac{N}{N-2}$ and assume that \eqref{mt1} holds true. 
Then the blow-up rate of the diffusion-induced blowing solution predicted by Theorem \ref{thmbu} is determined by
\be
\Vert u(\cdot, t)\Vert_\infty \ \approx \ (T_\de-t)^{-\frac{1}{p-1}}, \quad t\uparrow T_\de\label{ik2}.
\ee
\end{theorem}


The paper is organized as follows. Section \ref{prr} introduces some preliminary results on problem \eqref{fkpp}-\eqref{id}. Section \ref{blrn} contains the proof of our main blow-up Theorem \ref{thmbu} and that of the completeness of blow-up given by Corollary \ref{nik2}. Section \ref{blp} discusses  the exact blow-up rate  provided by Theorem \ref{tbu}. In section \ref{blp} we also identify the blow-up profile of solution $u,$  and thus we determine the form of Turing instability patterns occurring as a consequence of the diffusion-driven instability.
\section{Preparatory results}\label{prr}
\ In the current subsection we present some key properties for the $u(x,t)$ solution of  \eqref{fkpp}-\eqref{id}
We first point that the existence of a unique classical local-in-time solution of the non-local problem \eqref{fkpp}-\eqref{id} can be established by using results existing in \cite{QS} (see Remark 51.11 and Example 51.13 ) and in \cite{s98}.

Henceforth, we use the notation $C$ and $C_i, i=1,\dots,$ to denote positive constants.

Next we provide a result that establishing the positivity of solutions of \eqref{fkpp}-\eqref{id} once non-negative initial data are considered.
\begin{lemma}\label{pos}
Let consider initial date $u_0\in L^{\beta_0}(\Om)$ with $\beta_0=\max\{\beta,2\},$ $u_0(x)\geq 0$ in $\Omega,$ then
\bean
u(x,t)\geq 0, \quad\mbox{for any}\quad (x,t) \in \ol{Q}_T,
\eean
where $Q_T:=\Omega \times (0,T).$ 
\end{lemma}
\begin{proof}
Set $u^{-}:=-\min\{u,0\}\geq 0,$ then by the assumption on the initial data we have $u_0^{-}=0$ and thus
\be\label{ik1}
\int_{\Omega} (u_0^{-})^2\,dx=0.
\ee
 Next by testing \eqref{fkpp} by $u^{-}$ we derive
\bea
  \frac{1}{2} \frac{d}{dt} \int_{\Omega} (u^{-})^2\,dx && =-\int_{\Omega} |\nabla u^{-}|^2\,dx+\int_{\Omega}|u|^{p-1}(u^{-})^2\,dx\left(1-\sigma\avint |u|^{\beta-1}u\;dx\right)\no\\
&& \leq \int_{\Omega}|u|^{p-1}(u^{-})^2\,dx\left(1+\sigma\avint |u|^{\beta}\;dx\right)\no\\
&& \leq C(T)\int_{\Omega}(u^{-})^2\,dx,\label{ik2a}
\eea
where
\[
C(T):=\left[M^{p-1}(T) \left(1+\sigma M^{\beta}(T)\right)\right]<\infty,
\]
since $M(T):=\max_{(x,t)\in Q_T} |u(x,t)|<+\infty$ for a classical solution of \eqref{fkpp}-\eqref{id}.

Inequality $\eqref{ik2a}$ by virtue of \eqref{ik1} entails $\int_{\Omega} (u^{-})^2\,dx=0,$ and thus $u(x,t)\geq 0$ in $Q_T.$

\end{proof}
Due to the above positivity result, henceforth we focus on the investigation of the problem

\bea
 && u_t-\Delta u=u^{p}\left(1-\sigma\avint u^{\beta}\;dx\right), \quad\mbox{in}\quad  Q_T,\label{pfkpp}\\
  && \frac{\partial u}{\partial \nu}=0, \quad\mbox{on}\quad \Gamma_T:=\partial \Omega \times (0,T),
  \\
 && u(x,0)=u_0(x)\geq 0, \quad x\in\Om.\label{pid}
\eea


The next lemma clarifies  the evolution of the norm
$$m_\beta(t):=\int_{\Omega} u^{\beta}(x,t)\,dx,$$
 along a nontrivial solution of \eqref{pfkpp}-\eqref{pid}.
\begin{lemma}\label{Lemma:mbetaestimate} Let $u$ be a solution of \eqref{pfkpp}-\eqref{pid} with $u_0\in L^\beta(\Om).$  
 If $\beta>1$ there holds
\begin{itemize}
\item[$\rm (i)$] $0<m_\beta(0)\leq1/\sigma$ implies $m_\beta(t)<1/\sigma$ for all $t\in(0,T]$, and
\item[$\rm (ii)$] $m_\beta(0)\geq1/\sigma$ implies $m_\beta(t)<m_\beta(0)$ for all $t\in(0,T]$.
\end{itemize}
\end{lemma}
\begin{proof}
A direct calculation and by virtue of \eqref{pfkpp} implies
\bea\label{evo_m_beta}
 m'_\beta(t)
  &&=
   -4\frac{\beta-1}{\beta}\intom\big|\nabla u^{\beta/2}\big|^2\dx
    +
     \beta\left(1-\sigma m_\beta(t)\right)m_{p+\beta-1}(t)\no\\
		&& <
   \beta\big(1-\sigma m_\beta(t)\big)m_{p+\beta-1}(t),\quad\mbox{for any}\quad 0<t<T,\label{EVO}
\eea
using also the fact $\beta>1.$
Under the assumption $m_\beta(0)\leq1/\sigma$,  by \eqref{EVO} we infer that there cannot be time $t_{\si}>0$ such that $m_\beta(t_{\si})=1/\sigma$ and $m'_\beta(t_{\si})\geq0$. Thus $m_\beta(t)\leq1/\sigma$ for all $t\in(0,T]$, and in the fact strict inequality follows. Namely, if $m_\beta(\hat{t}_{\si})=1/\sigma$ for some $\hat{t}_{\si}\in (0,T)$, then $m'_\beta(\hat{t}_{\si})<0,$ due to \eqref{EVO}, which infers that $m_{\beta}(t)$ would have thus exceeded $1/\sigma$ at some previous time $t'\in(0,\hat{t}),$ leading to a contradiction. Then an identical argument to $(i)$  implies $(ii).$
\end{proof}
\begin{rem}
An immediate consequence of Lemma \ref{Lemma:mbetaestimate} $(i)$ is a lower estimate of the average of solution $u$ over domain $\Omega.$ Indeed, under the assumption $0<m_\beta(0)\leq1/\sigma,$ which actually guarantees that
\bgee
K(t)\geq 0,\quad \mbox{for any}\quad 0<t<T,
\egee
 then averaging  \eqref{pfkpp} over $\Omega$ entails
\bean
\frac{d}{dt} \avint u\,dx\geq 0,\quad \mbox{for any}\quad 0<t<T,
\eean
in conjunction with Lemma \ref{pos}, which finally implies
\bea\label{as2}
\overline{u}(t):=\avint u\,dx\geq \ol{u}_0>0,\quad \mbox{for any}\quad 0<t<T,
\eea
since $u_0\in L^\beta(\Om)\setminus\{0\}.$
\end{rem}

 The global existence of positive classical solutions was proven in \cite{BChL,B}, yet for the sake of completeness we state these results in the sequel.

 \begin{theorem}\label{thm1}{\cite{BChL,B}}
Let $\beta\geq1$, and assume that $u_0$ is non-negative with  $u_0\in L^k(\Omega)$ for $1<k<+\infty.$
Assume further that $p$ satisfies
$$
1<p<1+\left(1-\frac{2}{q}\right)\beta,
$$
where
$$
q=
\begin{cases}
\frac{2N}{N-2},\ N\geq3,\\
2<p<+\infty,\ N=2,\\
\infty,\ N=1,
\end{cases}
$$
then there exists a unique non-negative classical global-in-time solution to \eqref{pfkpp}-\eqref{pid}.
\end{theorem}
\begin{rem}
Note that Theorem \ref{thm1}  for $N\geq 3$  guarantees the existence of global-in-time solutions of problem \eqref{fkpp2a}-\eqref{id2} in the range $1<p<1+\frac{2}{N}\beta$ and  for any $\beta\geq 1.$ In particular, choosing $\beta>\frac{N}{N-2}$ we obtain  global-in-time solutions for $1<p<\frac{N}{N-2},$ while on the other hand,  if $p>\frac{N}{N-2}$ then Theorem \ref{thmbu} establishes finite-time blow-up.  Consequently, for the specific choice $\beta>\frac{N}{N-2}$ Theorems \ref{thm1} and \ref{thmbu} provide an optimal result regarding the long-time behaviour of the solution to
\eqref{fkpp2a}-\eqref{id2}, although is still unclear what happens in the critical case $p=\beta=\frac{N}{N-2}.$ Nevertheless, for  $\beta< \frac{N}{N-2}$ our approach still works but leaves a gap between regarding the existence of  global-in-time and blowing up solutions  in the interval $ p\in( 1+\frac{2}{N}\beta,\frac{N}{N-2}).$
\end{rem}
We also have the following result providing the asymptotic behaviour of the solution in the case $\beta=1$,
\begin{theorem}\label{thm2}{\cite{BChL,B}}
Let $u$ be a non-negative classical solution obtained
from Theorem \ref{thm1}, $v$ be the solution to the heat equation
with Neumann boundary condition and initial data $\int_\Omega
v_0(x)dx= m_0$, then,
\begin{align}
\|u(\cdot,t)-v(\cdot,t)-(1-m_0)\|_{L^2(\Omega)}\leq C_1e^{-C_2t},
\end{align}
where $C_1,C_2$ are constants depending on the initial mass $m_0$ and $\|u_0\|_{L^{2 p}(\Omega)}$.
\end{theorem}

\section{Main results}\label{blrn}

\subsection{Diffusion-driven blow-up}
The current subsection is devoted to the proof of the occurrence of a {\it diffusion-driven blow-up (DDBU)} for the solution of problem  \eqref{fkpp2a}-\eqref{id2} under spiky initial data of the form \eqref{sid}-\eqref{fd}. Accordingly an approach, previously used in \cite{hy95,KS16, KS18, LN09}, will be implemented but it should be modified accordingly due to the form of the non-local term. However in order to proceed further we first need to establish some auxiliary results.

For the function $\phi_{\delta}$ defined by \eqref{fd} we set
\bge
&&\alpha_1=\sup_{0<\de<1}\frac{1}{\ol{\phi}_{\de}^{\mu} }\left(\avintr \phi_{\de}^p\,dx\right),\label{kl1}\\
&&\alpha_2=\inf_{0<\de<1}\frac{1}{\ol{\phi}_{\de}^{\mu} }\left(\avintr \phi_{\de}^p\,dx\right),\label{kl2}
\ege
for
\bge\label{mmnk1}
\mu:=\displaystyle{\frac{p\,\ell}{k-1}}>p,
\ege
and some $0<k<p$ such that $N>\frac{2p}{k-1}.$

Then the first auxiliary result presents the key properties of $\phi_{\delta}.$
\begin{lemma}[$\phi_{\de}-$properties]\label{PhiDeltaProp}
Let
\bea\label{st6}
p>\frac{N}{N-2},\quad N\geq 3,
\eea then the  function $\phi_{\de}$ defined  by \eqref{fd} satisfies the following:
\begin{enumerate}[label=(\roman*)]
\item  There holds that
\begin{equation}\label{Deltaphi}
	\Delta_r \phi_{\de}\geq-N a \phi_{\de}^p,
\end{equation}
 in a weak sense for any  $\delta\in(0,1),$  where $\Delta_r:= \frac{\partial^2}{\partial r^2}+\frac{N-1}{r}\frac{\partial }{\partial r}.$
\item  If $\zeta>0$ and $N>\zeta a$ then
\begin{equation}\label{nonlocalterm}
 \frac{1}{|B_1|}\int_{B_1} \phi_{\de}^\zeta(x)\,dx
:=\avintr \phi_{\de}^{\zeta}(x)\,dx=\,N\,\omega_{N}\int_0^{1} \,r^{N-1}\phi_{\de}^{\zeta}(r)\,dr=\frac{N}{N- a \zeta}+O(\delta^{N- a \zeta}),\quad\mbox{as}\quad \delta\downarrow 0,
\end{equation}
where  $\omega_N:=|B_1|=\pi^{N/2} \Gamma (N/2)$ is the volume of the unit ball in $\mathbb{R}^N.$
\item
Choose now  parameter $\ell$ so that
\bge\label{st7}
k-1<\ell<\frac{N(p-1)}{2p},
\ege
then for the quantities  $\alpha_1,\alpha_2$ defined by \eqref{kl1} and \eqref{kl2} are bounded  thanks  to \eqref{nonlocalterm} and \eqref{st6}.

Consider also
\bge\label{kl4}
d=d(\la, \sigma):=\lambda-\sigma2^{\beta(\mu+1)/p} a_1^{\beta/p} \Lambda_1^{\beta\mu/p} \lambda^{\beta+1},
\ege
which is a positive constant for any $0<\lambda\ll 1$ since $0<\Lambda_1:=\sup_{0<\delta<1} \bar{\phi_{\de}}<\infty$ thanks to \eqref{st7}.

Then there exists some $\lambda_0:=\lambda_0(\sigma,\de)>0$  small enough such that
\begin{align}\label{supercritical}
\Delta_r u_0+d\lambda^{-1}u_0^p	\geq 2 u_0^p,
\end{align}
for any  $0<\frac{\lambda_0}{2}<\lambda<\lambda_0.$
\end{enumerate}
\end{lemma}
\begin{proof} For the proof of  $(i)$ and $(ii)$ see \cite{hy95, KS16, KS18}, hence it  remains to prove $(iii).$ Recall that $u_0=\lambda \phi_{\delta}$ and if we  fix some $0<\lambda_0 <1$ then
\begin{align*}
	\Delta_r u_0+d\lambda^{-1}u_0^p =
	\lambda\Delta_r\phi_{\de}+d\lambda^{p-1}\phi_{\de}^p\geq \lambda\Delta_r\phi_{\de}+d_0\lambda^{p-1}\phi_{\de}^p,
	\end{align*}
where $d_0:=\inf_{\la\in(\frac{\la_0}{2},\lambda_0)} d(\la),$ and so it is enough to prove that
	\begin{align*}
	\lambda\Delta_r\phi_{\de}+d_0\lambda^{p-1}\phi_{\de}^p
	&\geq
2\lambda^p\phi_{\de}^p.
	\end{align*}
Note that thanks to  \eqref{Deltaphi} it is sufficient to show
	\begin{eqnarray*}
	-\lambda N a \phi_{\de}^p+d_0\lambda^{p-1}\phi_{\de}^p
	\geq
2\lambda^p\phi_{\de}^p,
\end{eqnarray*}
or equivalently
\bean
	d_0\lambda^{p-2}
	\geq
2\lambda^{p-1}+N a,
\eean
which is finally true for $0< \la_0\ll 1.$
\end{proof}
Given $0<\delta<1$, we recall that  $T_\delta>0$  is the maximal existence time for the solution to \eqref{fkpp2a}-\eqref{id2} with initial data $u_0=\lambda\phi_{\de}.$
Henceforth we consider   $0<\lambda<\lambda_0$ so that Lemma \ref{PhiDeltaProp} is valid.

The next result provides a useful point estimate for any $0<r<1$ of the radial symmetric solution $u(r,t)$ in terms of its average over $B_1.$
\begin{lemma}\label{NewProblProp}
	For any $0<r\leq 1$ there holds
	\begin{equation}\label{as0}
		r^N u(r,t)\leq \overline{u}(t):=\avintr u(x,t)\,dx=N\int_0^{1} \,y^{N-1}u(y,t)\,dr
	\end{equation}
	and
	\begin{equation}\label{as1}
		u_r\left(\frac34,t\right)\leq-c,\quad 0\leq t<T_\delta.
	\end{equation}
\end{lemma}
\begin{proof}
We first define the  operator
$$
\mathcal{H}[w]:=
w_t-w_{rr}+\frac{N-1}{r}w_r-pu^{p-1}K(t)w
$$
with $w=r^{N-1}u_r$ and then we note that
\begin{eqnarray}
&&\mathcal{H}[w]=0,\;0<r<1,\;\; 0<t<T_{\de},\label{mmnk2}\\
&&w(r,t)=0,\quad r=0,1,\quad 0<t<T_{\de},\label{mmnk3}\\
&& w(r,0)<0,\quad 0<r<1, \label{mmnk4}
\end{eqnarray}
following similar calculations as in \cite{hy95}. 

The maximum principle, treating  \eqref{mmnk2}-\eqref{mmnk4} as a local problem, then implies that $w\leq 0$, thus $u_{r}\leq0$ for $(r,t)\in(0,1)\times(0,T_{\de})$ and so
$$
\overline{u}:=N\int_0^{1} \,y^{N-1}u(y,t)\,dy
 \geq
 N \int_0^{r} \,y^{N-1}u(y,t)\,d\sigma
\geq
N u(r,t)\,\int_0^r y^{N-1}\,d \sigma
=
u(r,t)\,r^N,
$$
for any $0<r<1$ and $0<t<T_{\de},$ recalling that $\omega_{N}=|B_1(0)|.$

By virtue of  Lemma \ref{Lemma:mbetaestimate}  and for a classical solution $u$ of \eqref{fkpp2a}-\eqref{id2}, we obtain that the term $p K(t) u^{p-1}$, that is the coefficient of the linear term in $\mathcal{H}[w]$, is uniformly bounded in $\frac{1}{2}<r<1,\; 0<t<T_{\de}$ for all $0<\de<\de_0.$ Furthermore,  we have $p K(t) u^{p-1}w\leq 0$ due to Lemma \ref{pos} and Lemma \ref{Lemma:mbetaestimate}.  Next we compare $w$ with the solution of local problem
\begin{eqnarray}
&&\theta_t-\theta_{rr}+\frac{N-1}{r}\theta_r=0,\;\frac{1}{2}<r<1,\;\; 0<t<T_{\de},\\
&&\theta(r,t)=0,\quad r=\frac{1}{2},1,\quad 0<t<T_{\de},\\
&& \theta(r,0)=w(r,0)<0,\quad \frac{1}{2}<r<1,
\end{eqnarray}
to obtain that $w\leq \theta\leq 0$ in $(\frac{1}{2},1)\times(0,T_{\de})$ in conjunction with maximum principle.

In particular we have
\bean
 u_r\left(\frac{3}{4},t\right)\leq \left(\frac{4}{3}\right)^{N-1}\,\theta\left(\frac{3}{4},t\right)\leq -c,\quad 0<t<T_{\de},
\eean
where $c$ is independent of $0<\de<\de_0.$

\end{proof}
Next we prove an essential two-side $L^p-$estimate for the solution of \eqref{fkpp2a}-\eqref{id2}, inspired by an analogous result holding for the shadow system of Gierer-Meinhardt model, see also \cite[Proposition 8.1]{KS16} or \cite[Chapter 5, Proposition 5.3]{KS18}.
\begin{proposition}\label{prop8.1}
There exist $0<\delta_0<1$and $0<t_0\leq1$ independent of any $0<\delta\leq\delta_0$, such that the following estimate holds
\be\label{LpEst}
\frac12A_2\overline{u}^\mu\;dx
\leq\avintr u^p\;dx
\leq 2A_1\overline{u}^\mu\;dx,\quad
\mbox{for any}\quad t\in\left(0,\min\{t_0 ,T_\delta\}\right),
\ee
where
$
\mu
$
is given by \eqref{mmnk1}.  The positive constants $A_1$ and $A_2$ in \eqref{LpEst} are given by
\bean
&&A_1=\sup_{0<\de<1}\frac{1}{\ol{u}_0^{\mu} }\left(\avintr u_0^p\,dx\right)=\lambda^{p-\mu}\alpha_1,\\
&&A_2=\inf_{0<\de<1}\frac{1}{\ol{u}_0^{\mu} }\left(\avintr u_0^p\,dx\right)=\lambda^{p-\mu}\alpha_2,
\eean
and are bounded due to Lemma \ref{PhiDeltaProp}.
\end{proposition}
\begin{proof}
 For any $0<\de<\de_0,$ consider  $[0,t_0(\de)]$ to be the maximal time interval for which \eqref{LpEst} holds. Obviously there holds $0<t_0(\de)\leq T_{\de}$ for each $0<\de<\de_0.$ In case $t_0\geq 1$ there is nothing to prove since the statement \eqref{LpEst} automatically holds by  simply choosing $t_0=1.$ Hence  in the following we now assume that $t_0\leq 1.$

Integration of \eqref{fkpp2a} over $B_1,$  by virtue of \eqref{LpEst}, entails
\begin{align}
\frac{d\overline{u}}{dt}=\avint u^p\,dx\, \left(1-\sigma\avint u^\beta\,dx\right)\leq	\avint u^p\,dx\leq 2A_1\overline{u}^\mu,
\end{align}
and thus,
\begin{align}
\overline{u}\leq
\left[
\frac{1}{\overline{u}_0^{\;1-\mu}-2A_1(\mu-1)t}
\right]^{\frac{1}{\mu-1}}, \quad\mbox{for any}\quad t\in (0,t_0).
\end{align}

It can be also verified that
$$
\left[
\frac{1}{\overline{u}_0^{\;1-\mu}-2A_1(\mu-1)t}
\right]^{\frac{1}{\mu-1}}\leq2\overline{u}_0,
$$
provided that
$$
t\leq \min\left\{\frac{2^{\mu}-1}{2^{\mu}A_1(\mu-1)\overline{u}_0^{\;\mu-1}},\frac{\overline{u}_0^{\;\mu-1}}{2 A_1(\mu-1)}\right\}.
$$
Consequently, we deduce that
\begin{equation}\label{zUB}
\overline{u}(t)	\leq2\overline{u}_0\leq 2\Lambda:=2\sup_{\delta\in(0,\delta_0)}\overline{u}_0,
\end{equation}
when $0<t<t_2:=\min\left\{t_0,t_1\right\},$ and for
\begin{equation}
t_1:=\left\{\frac{2^{\mu}-1}{2^{\mu}A_1(\mu-1)\Lambda^{\;\mu-1}},\frac{\Lambda^{\;\mu-1}}{2 A_1(\mu-1)}\right\},	
\end{equation}
which is independent of $0<\delta<\delta_0$.

Next, for given $\varepsilon>0$, we define  the auxiliary function
$$
\chi:=r^{N-1} u_{r}+\varepsilon\,r^{N}\,\frac{u^k}{\overline{u}^{\ell}},
$$
recalling thatexponents $k$ and $\ell$ are defined in \eqref{mmnk1}  and \eqref{st7}.

It is readily seen, cf. \cite{hy95},  that
\begin{equation}\label{ps1}
\mathcal{H}\left[r^{N-1} u_{r}\right]=0,
\end{equation}
while by straightforward calculations we derive
\begin{align}
\mathcal{H}\left[\varepsilon r^{N} \frac{u^k}{\overline{u}^{\ell}}\right]&=\frac{2k(N-1)\varepsilon r^{N-1}u^{k-1}}{\overline{u}^{\ell}}u_{r}+\frac{k\varepsilon r^N u^{p-1+k}}{\overline{u}^{\ell}} K(t)-\frac{\ell \varepsilon r^N u^k}{\overline{u}^{\ell+1}}\,K(t)\,\avintr u^p\,dx
\nonumber\\
&-\frac{2 k N\varepsilon r^{N-1} u^{k-1}}{\overline{u}^{\ell}}u_r-\frac{k(k-1)\varepsilon r^N u^{k-2}}{\overline{u}^{\ell}}u_r^2-\frac{p\varepsilon r^N u^{p-1+k}}{\overline{u}^{\ell}} K(t)
\nonumber\\
&\leq
 -\frac{2 k \varepsilon r^{N-1} u^{k-1}}{\overline{u}^{\ell}}u_r
 -\frac{\ell \varepsilon r^N u^k}{\overline{u}^{\ell+1}}\,K(t)\,\avintr u^p\,dx
 -\frac{(p-k)\varepsilon r^N u^{p-1+k}}{\overline{u}^{\ell}} K(t)
 \nonumber\\
&=
 -\frac{2 k \varepsilon u^{k-1}}{\overline{u}^{\ell}}\chi
 +\frac{2 k \varepsilon^2 r^{N} u^{2k-1}}{\overline{u}^{2\ell}}
 -\frac{\varepsilon r^Nu^k}{\overline{u}^{2\ell}}
 \left[
 K(t)\ell \overline{u}^{\ell-1}\avintr u^p\,dx
 +(p-k) u^{p-1}\overline{u}^{\ell} K(t)
 \right]
  \nonumber\\
&=
 -\frac{2 k \varepsilon u^{k-1}}{\overline{u}^{\ell}}\chi
 +\frac{\varepsilon r^Nu^k}{\overline{u}^{2\ell}}
 \left[
 2k\varepsilon u^{k-1}
 -K(t)\ell \overline{u}^{\ell-1}\avintr u^p\,dx
 -(p-k) u^{p-1}\overline{u}^\ell K(t)
 \right]
 \nonumber\\
 &=
  -\frac{2 k \varepsilon u^{k-1}}{\overline{u}^{\ell}}\chi
+I.
 \label{hest1}
\end{align}
Next we show that $$I:=  2k\varepsilon u^{k-1}
 -K(t)\ell \overline{u}^{\ell-1}\avintr u^p\,dx
 -K(t)(p-k) u^{p-1}\overline{u}^\ell\leq0.$$
Indeed,  using  \eqref{mt1} in conjunction with Jensen's inequality and \eqref{LpEst}, \eqref{zUB} we immediately derive
\bea\label{mt2}
K(t)&&\geq 1-\sigma\left(\avintr u^p dx\right)^{\beta/p}\no\\
&& \geq 1-\sigma2^{\beta(\mu+1)/p} A_1^{\beta/p} \Lambda^{\beta\mu/p}\no\\
&&\geq 1-\sigma2^{\beta(\mu+1)/p} a_1^{\beta/p} \Lambda_1^{\beta\mu/p} \lambda^{\beta}:=D,\quad\mbox{for any}\quad 0<t<\min\{t_0,t_1\},
\eea
recalling that  $0<\Lambda_1=\sup_{0<\delta<1} \bar{\phi_{\de}}<\infty.$
Notably we have
\bge\label{kl5}
D=  1-\sigma2^{\beta(\mu+1)/p} a_1^{\beta/p} \Lambda_1^{\beta\mu/p} \lambda^{\beta}=\lambda^{-1} d,
\ege
 is a positive  constant for  $0<\lambda<\lambda_0\ll1,$  depending on $u_0$ but not on $0<\delta<\delta_0;$  here we  recall  that  $d$  is defined by \eqref{kl4}.

Next combining \eqref{hest1} with \eqref{LpEst}, \eqref{zUB} and \eqref{mt2} we deduce
\begin{align}
I& \leq 2k\varepsilon u^{k-1}
 -\frac{A_2 D\ell }{2} \overline{u}^{\mu+\ell-1}
 -D(p-k) u^{p-1}\overline{u}^\ell\nonumber\\
& \leq 2k\varepsilon u^{k-1}
 -D(p-k) (2 \Lambda)^{\ell} u^{p-1}, \quad\mbox{for any}\quad 0<t<t_0.
\label{I1}
\end{align}
Then \eqref{I1} in conjunction with Young's inequality  leads to $I\leq 0,$ by also choosing $\vep$ sufficiently small and independent of $0<\de<\de_0.$

Finally combining \eqref{ps1} with \eqref{hest1} we obtain
\bean
\mathcal{H}[\chi]\leq -\frac{2 k \vep u^{k-1}}{\ol{u}^{\ell}}\chi\quad\mbox{in}\quad \left(0,\frac{3}{4}\right)\times(0,t_2),
\eean
for any  $\vep$ sufficiently small and independent of $0<\de<\de_0.$

Note that  $\chi(0,t)=0,$ whilst at $r=\frac{3}{4}$ due to Lemma \ref{NewProblProp} and \eqref{zUB} there holds
\begin{align}
\chi\left(\frac{3}{4}\right)
&=
\left(\frac34\right)^{N-1}u_r\left(\frac34,t\right)
+
\varepsilon\left(\frac34\right)^{N}\frac{u^k\left(\frac{3}{4},t\right)}{\overline{u}^\ell}
\nonumber\\
&\leq
-c\left(\frac34\right)^{N-1}
+
\varepsilon\left(\frac34\right)^{N-kN}\overline{u}^{k-\ell}(t)
\nonumber\\
&\leq
-c\left(\frac{3}{4}\right)^{N-1}
+\varepsilon\left(\frac{3}{4}\right)^{N-kN}(2\Lambda)^{k-\ell}
<0,
\end{align}
for $\varepsilon$ sufficiently small.

Subsequently for $t=0,\; 0<r<\delta$ and fixed $0<\lambda\ll 1$ we calculate,
\begin{align*}
\chi(r,0)
&=r^{N-1}\left[\lambda \phi'_{\de}(r)+\varepsilon r\lambda^{k-\ell}\frac{\phi_{\de}^k}{\overline{\phi}_{\de}^{\ell}}\right]\\
&\leq
 r^{N}\left[-\frac{\lambda a }{\de^{a+2}}
 +\varepsilon\lambda^{k-\ell}
 \frac{(1+\frac{a}{2})^k}{\de^{ak}\left(\frac{N}{N-a\beta}
 +O\left(\de^{N-a\beta}\right)\right)}\right]\\
&\leq r^{N}\left[-\frac{\lambda a }{\de^{a+2}}+\varepsilon\lambda^{k-\ell}\frac{(1+\frac{a}{2})^k}{\de^{ak}}\right]<0,
\end{align*}
since $a+2=ap>ak$ and for $\varepsilon$ small enough and independent of $0<\de<\de_0.$

On the other hand, for $t=0,\ \delta<r<\frac{3}{4}$ and fixed $0<\lambda \ll 1$ we have,
\begin{align*}
\chi(r,0)&=r^{N-1}\left[-\frac{\lambda\,a}{r^{a+1}} +\varepsilon r\lambda^{k-\ell}\frac{1}{r^{ak}\left(\frac{N}{N-a\beta}+O\left(\de^{N-a\beta}\right)\right)}\right]\\
&\leq r^{N-1}\left[-\frac{\lambda\,a}{r^{a+1}} +\varepsilon \lambda^{k-\ell}\frac{1}{r^{ak-1}}\right]<0,
\end{align*}
since $a+1>ak-1$ and again taking $\varepsilon$ small enough and still independent of $0<\de<\de_0.$

Outlining  we have
\bean
&&\mathcal{H}[\chi]\leq -\frac{2 k \varepsilon u^{k-1}}{\overline{u}^{\ell}}\chi,
\quad\mbox{for}\quad 0<r<\frac{3}{4}\quad\mbox{and}\quad 0<t<t_2,
\\
&&\chi\leq 0,\quad\mbox{for}\quad r=0,\frac{3}{4}\quad\mbox{and}\quad  0<t<t_2,
\\
 &&\chi\leq 0,\quad\mbox{for}\quad 0<r<\frac{3}{4}\quad\mbox{and}\quad  t=0,
\eean
hence maximum principle entails  $\chi\leq 0,$ that is
$$
u_r\leq \varepsilon r\frac{u^k}{\overline{u}^{\ell}},
$$
which by integrating over $r\in(0,\frac34)$  leads to
\bea
\label{tbsd10}
u(r,t)\leq\left[\frac{2 \overline{u}^{\ell}}{\varepsilon (k-1)}\right]^{\frac{1}{k-1}} r^{-\frac{2}{k-1}}\quad\mbox{for}\quad 0<r<\frac{3}{4}\quad\mbox{and}\quad 0<t<t_2.
\eea
Hence \eqref{tbsd10} implies that for any $0<r<\frac{3}{4}$
$$
	\frac{1}{|B_1|}\int_{B_r(0)} u^p\,dx\leq N\left[\frac{2}{\varepsilon (k-1)}\right]^{\frac{p}{k-1}}\frac{r^{N-\frac{2p}{k-1}}}{N-\frac{2p}{k-1}}\overline{u}^{\mu},\quad 0<t<t_2,
$$
and by choosing $r$ sufficiently small, recalling that $N>\frac{2p}{k-1},$ we end up with the following estimate
\be\label{tbsd11}
\frac{1}{|B_1|}\int_{B_r(0)} u^p\,dx\leq \frac{A_2}{8} \overline{u}^{\mu},
\ee
for all $0<t<t_2$ since $\mu=\frac{p\,\ell}{k-1}.$

Next we set
$
\psi:=\frac{u}{\overline{u}^{\nu}},
$
for $\nu:=\frac{\mu}{p}=\frac{\ell}{k-1}>1$, then we can easily check that $\psi$ satisfies the following non-local equation
$$
\psi_t
=
\Delta \psi
+\left(
\frac{K(t)}{\overline{u}^{\nu}}u^p
-
\nu\frac{K(t)u}{\overline{u}^{\nu+1}} \avintr u^p\;dx
\right).
$$
We easily observe that by virtue of Lemma \ref{NewProblProp}  and relations \eqref{as2}, \eqref{LpEst}  and \eqref{zUB} the terms
\bean
\frac{K(t)}{\overline{u}^{\nu}}u^p,\quad\mbox{and}\quad \nu\frac{K(t)u}{\overline{u}^{\nu+1}} \avintr u^p\;dx,
\eean
are uniformly bounded in $[B_1(0)\setminus B_r(0)]\times \left(0,\min\{t_0(\de),t_1\}\right).$

Then standard parabolic regularity theory, \cite{LSU}, guarantees the existence of a time $t_3>0$ independent of $0<\de<\de_0$ such that
\bea\label{tbsd12}
\left|\frac{1}{|B_1|}\int_{B_1\setminus B_r(0)}\frac{u^p}{\ol{u}^\mu}\,dx-\frac{1}{|B_1|}\int_{B_1\setminus B_r(0)}\frac{u_0^p}{\ol{u}_0^\mu}\,dx\right|<\frac{A_2}{8},
\eea
for $0\leq t\leq \min\{t_0(\de), t_2,t_3\}.$

Considering that for some $\de_1\in(0,\de_0)$ there holds that $t_0(\de_1)\leq\min\{t_2,t_3, T_{\de_1}\}$ and then by virtue of \eqref{tbsd11} and \eqref{tbsd12} we deduce
\bean
&&\left|\frac{1}{|B_1|}\int_{B_1}\frac{u^p}{\ol{u}^\mu}\,dx-\frac{1}{|B_1|}\int_{B_1}\frac{u_0^p}{\ol{u}_0^\mu}\,dx\right|\\
&&\leq \left|\frac{1}{|B_1|}\int_{B_r(0)}\frac{u^p}{\ol{u}^\mu}\,dx-\frac{1}{|B_1|}\int_{B_r(0)}\frac{u_0^p}{\ol{u}_0^\mu}\,dx\right|\\
&&+\left|\frac{1}{|B_1|}\int_{B_1\setminus B_r(0)}\frac{u^p}{\ol{u}^\mu}\,dx-\frac{1}{|B_1|}\int_{B_1\setminus B_r(0)}\frac{u_0^p}{\ol{u}_0^\mu}\,dx\right|\\
&&\leq \frac{3 A_2}{8},
\eean
and thus
\bea\label{tbsd13}
\frac{11 A_1}{8}\leq \frac{1}{|B_1|}\int_{B_1} \frac{u^p}{\ol{u}^{\mu}}\,dx\leq \frac{5 A_2}{8}\quad\mbox{for any}\quad 0<t<\min\{t_1,t_0(\de_1)\}.
\eea
We can then use continuity arguments in conjunction with \eqref{tbsd13} and the fact that  $0<t_0(\de_1)<T_{\de_1}$ to extend the validity of \eqref{LpEst} beyond $t_0(\de_1),$ which actually contradicts the definition of $t_0(\de_1).$

Eventually we obtain that \eqref{LpEst} as well as all the preceding estimations are valid for any $0<t<\min\{\widetilde{t}_0,T_{\de}\}$ for $\widetilde{t}_0=\min\{t_2,t_3\}.$ This competes the proof of the proposition.
\end{proof}
We now are ready to prove Theorem \ref{thmbu}, the main result in the current subsection.

\begin{proof}[Proof of Theorem \ref{thmbu}]
By virtue of  the  key estimate \eqref{mt2}, derived in the proof of Proposition \ref{prop8.1},  we can easily check that $u$ satisfies
\bean
u_t=  \Delta_r u +K(t)u^p\geq \Delta_r u + D u^p\quad\mbox{in}\quad B_1\times\left(0,\min\{t_0,T_{\de}\}\right),
\eean
reacalling  that $D$ depends on $u_0$ but not on $0<\de<\de_0.$ Thus by comparison principle (in terms of the heat operator) we infer
\bea \label{nik}
u(x,t)\geq \tilde{u}(x,t)\quad\mbox{in}\quad\bar{B}_1\times\left[0,\min\{t_0,T_{\de}\}\right], 
\eea
where $\tilde{u}$ solves the following local problem  problem
\bea
&&\tilde{u}_t= \Delta_r \tilde{u} + D \tilde{u}^p\quad\mbox{in}\quad B_1\times\left(0,\min\{t_0,T_{\de}\}\right),\label{lcp1}\\
&&\displaystyle\frac{\partial \tilde{u}}{\partial \nu}=0,\quad\mbox{on}\quad \partial B_1\times \left(0,\min\{t_0,T_{\de}\}\right),\label{lcp2} \\
&&\tilde{u}(x,0)=u_0(x), \quad\mbox{in}\quad B_1.\label{lcp3}
\eea
Consider now the auxiliary function $h:=\tilde{u}_t-\tilde{u}^p,$ then by straightforward calculations we deduce
\bean
h_t=\Delta_r h+p(p-1) \tilde{u}^{p-2} |\nabla \tilde{u}|^2+ D p \tilde{u}^{p-1}\,h\geq \Delta_r h+ D p \tilde{u}^{p-1}\,h \quad\mbox{in}\quad B_1\times\left(0,\min\{t_0,T_{\de}\}\right)
\eean
for $p>1$ and $\displaystyle{\frac{\partial h}{\partial \nu}}=0$ on $\partial{B}_1\times\left(0,\min\{t_0,T_{\de}\}\right).$
Additionally, by virtue of \eqref{supercritical} and \eqref{kl5}, we have 
\bean
h(x,0)=\Delta_r \tilde{u}(x,0)+D \tilde{u}^p(x,0)-\tilde{u}^p(x,0)=\Delta_r u_0+(D-1) u_0^p\geq u_0^p,\quad\mbox{in}\quad B_1.
\eean
Therefore maximum principle entails that $h>0$ in $\bar{B_1}\times\left[0,\min\{t_0,T_{\de}\}\right]$ and that is
\bean
\tilde{u}_t>\tilde{u}^p\quad\mbox{in}\quad \bar{B_1}\times\left[0,\min\{t_0,T_{\de}\}\right].
\eean
Integrating we derive
\bean
\tilde{u}(r,t)\geq\left(\frac{1}{u_0^{p-1}(r)}-(p-1)t\right)^{-\frac{1}{p-1}},\quad\mbox{in}\quad \bar{B_1}\times\left[0,\min\{t_0,T_{\de}\}\right],
\eean
which for $r=0$ reads
\bean
\tilde{u}(0,t)\geq\left(\frac{1}{u_0^{p-1}(0)}-(p-1)t\right)^{-\frac{1}{p-1}}=\left\{\frac{\de^{a(p-1)}}{\left[\la\left(1+\frac{a}{2}\right)\right]}-(p-1)t\right\}^{-\frac{1}{p-1}},
\eean
which entails finite time blow-up for $\widetilde{u},$ i.e.
\bean
||\tilde{u}(\cdot,t)||_{\infty}=\widetilde{u}(0,t)\to \infty\quad\mbox{as}\quad  t\to \tilde{T}_\de=\frac{1}{p-1}\left[\la \left(1+\frac{a}{2}\right)\right]^{1-p}\,\de^2,
\eean
and consequently finite-time blow-up for the solution $u$ of \eqref{fkpp2a}-\eqref{id2} at time $T_{\de}\leq \widetilde{T}_\de$ due to \eqref{nik}. Note also that $T_{\de}\to 0$ as $\de \to 0$ and thus the proof is complete.
\end{proof}
\begin{rem}
The finite-time blow-up established by Theorem \ref{thmbu} is actually a single-point blow-up, i.e. the solution $u(r,t)$ of  of \eqref{fkpp2a}-\eqref{id2} blows up only at the origin $r=0.$ Indeed, by virtue of \eqref{LpEst} and \eqref{zUB} we derive the following estimate
\bean
\avintr u(x,t)\,dx=N\int_0^{1} \,r^{N-1}u(r,t)\,dr\leq C<\infty,\quad\mbox{for any}\quad 0<t<T_\de,
\eean
wich in conjunction with \eqref{as0} implies that the blow-up set of $u$
\[
\mathcal{S}=\left\{r_0\in[0,1]:\quad\mbox{there exists}\quad r_n\to r_0\quad\mbox{and}\quad t_n\to T_{\de}: \lim_{n\to +\infty} u(r_n,t_n)=+\infty \right\}=\{0\}.
\]
\end{rem}
\subsection{Complete blow-up}
Interestingly the finite-time blow-up predicted by Theorem \ref{thmbu} for the solution $u$ of \eqref{fkpp2a}-\eqref{id2} is complete, roughly speaking there holds $u(x,t)=+\infty$ for any $x\in B_1$ and $t>T_{\de}.$
Before proving the latter result we need to provide an auxiliary result inspired by \cite{BC}, cf. \cite[Theorem 27.2]{QS}, and for which we will need some preliminary concepts.

Now set $f_k(V):=\min\{V^p,k\}, V\geq 0, k=1,2,\dots.$ and let $\tilde{u}_k$ be the solution of problem
\bean
&& V_t=\Delta_r V+f_k(V),\quad\mbox{in}\quad B_1\times(0,\infty),\\
&& \displaystyle\frac{\partial V}{\partial \nu}=0,\quad\mbox{on}\quad \partial B_1\times (0,\infty),\\
&& V(x,0)=u_0(x),\quad\mbox{in}\quad B_1.
\eean
It is easily seen that $\tilde{u}_k$ is globally defined and $\tilde{u}_{k+1}\geq \tilde{u}_k.$ Moreover $\tilde{u}_k$ solves the integral equation
\bea\label{ps50}
\tilde{u}_k(x,t)=\int_{B_1} G(x,y,t) u_0(x)\,dy+\int_0^t \int_{B_1} G(x,y,t-s) f_k(\tilde{u}_k(y,s))\,dy\,ds,
\eea
for any $x\in B_1,\; t>0$, where $G$ stands for the Neuman heat kernel in $B_1.$ Now since $G>0$ and $\tilde{u}_{k+1}\geq \tilde{u}_k$ if we pass to the limit into \eqref{ps50} we derivethen monotone convergence theorem implies
\bean
\bar{u}(x,t)=\int_{B_1} G(x,y,t) u_0(x)\,dy+\int_0^t \int_{B_1} G(x,y,t-s)\bar{u}^p(y,s))\,dy\,ds,\quad x\in B_1,\;t>0,
\eean
for $\bar{u}(x,t):=\lim_{k\to\infty}\tilde{u}_k(x,t)$ and where the double integral might be infinite. Clearly $\bar{u}(\cdot,t)=u(\cdot,t)$ for $t<\widetilde{T}_{\de}$ and if we set
\bean
T^{c}=T^{c}(u_0):=\inf\left\{t\geq \widetilde{T}_{\de}:\bar{u}(x,t)=\infty\quad\mbox{for all}\quad x\in B_1\right\},
\eean
then there holds $T^{c}(u_0)\geq \widetilde{T}_{\de}.$ Now we can provide a more rigorous definition of the complete blow-up.
\begin{definition}
We say that the solution of problem \eqref{lcp1}-\eqref{lcp3} blows up completely if $T^{c}(u_0)=\widetilde{T}_{\de}.$
\end{definition}

\begin{theorem}\label{cobu}
If $N\geq 3$ and $1<p<p_S:=\frac{N+2}{N-2}$  then solution of problem \eqref{lcp1}-\eqref{lcp3} exhibits a complete blow-up at $T_{\de}.$
\end{theorem}
\begin{proof}
For reader's convenience we split the  proof in several steps.\\
{\bf Step 1:} We claim that $\widetilde{u}_t\geq 0.$ Indeed, if we set $z=\widetilde{u}_t$ then $z,$ thanks to \eqref{supercritical} satisfies
\bean
&& z_t=\Delta_r z+D p \tilde{u}^{p-1} z,\quad\mbox{in}\quad B_1\times\left(0,\widetilde{T}_{\de}\right),\\
&& \displaystyle\frac{\partial z}{\partial \nu}=0,\quad\mbox{on}\quad \partial B_1\times \left(0,\widetilde{T}_{\de}\right),\\
&& z(x,0)=\widetilde{u}_t(x,0)=\Delta_r u_0(x)+D u_0^p(x)\geq 2u_0(x)\geq 0,\quad\mbox{in}\quad B_1,
\eean
and thus maximum principle verifies our claim.\\
{\bf Step 2:} In the current step we will prove that $||\widetilde{u}^p(\cdot,t)||_1\to +\infty$ as $t\to \widetilde{T}_{\de}-.$

Note that since $\widetilde{u}\geq 0$ and $\widetilde{u}_t\geq 0$ then the function $g: t\mapsto||\widetilde{u}^p(\cdot,t)||_1$ is nondecreasing.

Assume by contrary that $g$ is bounded then the $L^{k}-L^{\ell}-$estimates entail
\bea
\left\Vert e^{-tA}f \right\Vert_k\leq Cq(t)^{-\frac{N}{2}(\frac{1}{\ell}-\frac{1}{k})} e^{-\mu_2 t} \left\Vert f \right\Vert_\ell, \quad 1\leq \ell \leq k \leq \infty,
 \label{eqn:21}
\eea
for $t\geq 0$ and any $f\in L^{\ell}(\Om),$ where
\[ 0<q(t)=\min\{t,1\}\leq 1, \]
and the operator $A$  in \eqref{eqn:21} denotes $-\Delta$ provided with Neumann boundary condition, whilst  $\mu_2$ is  second eigenvalue of $A,$ see also \cite{henry, rothe}.

Now by virtue of the variation-of-parameters formula we deduce
\begin{eqnarray}
 \Vert \widetilde{u}(t)\Vert_k\leq C\Big( \Vert u_0\Vert_k +\int_{0}^te^{-\mu_2s} q(s)^{-\frac{N}{2}(\frac{1}{\ell}-\frac{1}{k})} ||\widetilde{u}^p(s)||_\ell \,ds\Big), \quad\mbox{for any}\quad 0<t<\widetilde{T}_\de,
 \label{eqn:27}
\end{eqnarray}
where integrability near $s=t$ of the integrand terms  appeared in \eqref{eqn:27}  is ensured under the condition
\bge\label{eqn:28}
\frac{N}{2}\Big(\frac{1}{\ell}-\frac{1}{k}\Big)<1.
\ege
Now for $\ell=1$ \eqref{eqn:27} in conjunction with our assumption gives
\begin{eqnarray}
 \Vert \widetilde{u}(t)\Vert_k\leq C\Big( \Vert u_0\Vert_k +\int_{0}^te^{-\mu_2s} q(s)^{-\frac{N}{2}(\frac{1}{\ell}-\frac{1}{k})} ||\widetilde{u}^p(s)||_1 \,ds\Big)\leq C(\widetilde{T}_\de),\;\mbox{for any}\; 0<t<\widetilde{T}_\de,\quad
 \label{eqn:27a}
\end{eqnarray}
provided that
\bge\label{ps20}
\frac{N}{2}\Big(1-\frac{1}{k}\Big)<1.
\ege
It is known, see \cite{BP, W}, that  for $N\geq 3$ and $k>\frac{N(p-1)}{2}$ the $L^{k}-$norm,of the solution of
\bgee
&&\xi_t = \Delta_r \xi + D \xi^p\quad\mbox{in}\quad B_1\times\left(0,\min\{t_0,T_{\de}\}\right),\\
&&\xi=0,\quad\mbox{on}\quad \partial B_1\times \left(0,\min\{t_0,T_{\de}\}\right), \\
&&\xi(x,0)=u_0(x), \quad\mbox{in}\quad B_1,
\egee
blows up in finite time, and thus by comparison arguments we also derive that
\bge
\Vert\widetilde{u}(t)\Vert_k\to +\infty\quad\mbox{as}\quad t\to \widetilde{T}_\de,\quad\mbox{for any}\quad k>\frac{N(p-1)}{2}.\label{ps10}
\ege
Since $1<p<p_S$ we can always find an exponent $k$ so that both \eqref{ps20} and \eqref{ps10} hold true, and thus we arrive at a contradiction due to \eqref{eqn:27a}.\\
{\bf Step 3:} Consider $\vep\in (0,1),$ then
\bean
\int_0^1 r^{N-1} \tilde{u}^p(r,t)\,dr&&=\int_0^{\vep} r^{N-1} \tilde{u}^p(r,t)\,dr+\int_\vep^{1-\vep} r^{N-1} \tilde{u}^p(r,t)\,dr+\int_{1-\vep}^1 r^{N-1} \tilde{u}^p(r,t)\,dr\\
&&:=I_1(\vep)+I_2(\vep)+I_3(\vep).
\eean
For $I_1(\vep)$ under the change of variable $r=\frac{\vep(R-\vep)}{1-2\vep}$ we derive
\bean
I_1(\vep)=\frac{\vep^N}{(1-2\vep)^N} \int_{\vep}^{1-\vep} (R-\vep)^{N-1} \tilde{u}^p(R,t)\, dR\leq \frac{\vep^N}{(1-2\vep)^N} I_2(\vep).
\eean
An estimate for $I_3(\vep)$ is obtained as follows
\bean
I_3(\vep)\leq \tilde{u}^p(1,t)\left(\frac{1-\left(1-\vep\right)^N}{N}\right)\leq \tilde{u}^p(1-\vep,t)\left(\frac{\left(1-\vep\right)^N-\vep^N}{N}\right)\leq I_2(\vep),
\eean
provided that $\vep$ is chosen small enough so that $1+\vep^N<2(1-\vep)^N, $ where also the fact that $\tilde{u}_r\leq 0$ for $r\in (0,1)$ has been taken into account. Consequently
\bea\label{ps30}
\int_{\vep}^{1-\vep}r^{N-1}\tilde{u}^p(r,t)\,dr\geq C(\vep) \int_{0}^{1}r^{N-1}\tilde{u}^p(r,t)\,dr, 
\eea
for $C(\vep):=2+\frac{\vep^N}{(1-2\vep)^N}.$

Set $v(r):=\lim_{t\to \tilde{T}_\de} \tilde{u}(r,t)\;\mbox{for any}\; r\in (0,1)$ then
\bea\label{ps40}
\int_{\vep}^{1-\vep} r^{N-1}v^p(r)\,dr=\lim_{t\to \tilde{T}_\de} \int_{\vep}^{1-\vep}r^{N-1}\tilde{u}^p(r,t)\,dr\geq \liminf_{t\to \tilde{T}_\de} C(\vep) \int_0^1 r^{N-1}\tilde{u}^p(r,t)\,dr=\infty,\qquad
\eea
where it has been successively used the monotone convergence of $\tilde{u}$ towards $v,$ relation \eqref{ps30}  and  Step 2. \\
{\bf Step 4:} Fix now some $r\in (0,1)$ and take some $t>\widetilde{T}_{\de}.$ Then we can find $\vep>0$ sufficiently small such that $t-\widetilde{T}_{\de}\geq 2\vep$ and $\vep<r<1-\vep.$ Next  by virtue of \eqref{ps50} and in conjunction with $f_k(\tilde{u}_k(R,s))\geq f_k(\tilde{u}_k(R,\widetilde{T}_\de))$ for $s\geq \widetilde{T}_\de$ we have
\bea\label{ps70}
\tilde{u}_k(r,t) &&\geq N \omega_N\int_0^t \int_{0}^1 R^{N-1} G(r,R,t-s) f_k(\tilde{u}_k(R,s))\,dR\,ds\no\\
&&\geq N \omega_N \tilde{C}(\vep)\int_{t-2\vep}^{t-\vep} \int_{\vep}^{1-\vep} R^{N-1}  f_k(\tilde{u}_k(R,s))\,dR\,ds\no\\
&&\geq N \omega_N\tilde{C}(\vep)\int_{t-2\vep}^{t-\vep} \int_{\vep}^{1-\vep} R^{N-1}  f_k(\tilde{u}_k(R,\widetilde{T}_\de))\,dR\,ds\no\\
&&\geq N \omega_N\vep\;\tilde{C}(\vep)\int_{\vep}^{1-\vep} R^{N-1}  f_k(\tilde{u}_k(R,\widetilde{T}_\de))\,dR,
\eea
where
\[
\tilde{C}(\vep):=\inf\left\{G(r,R,s): \vep<r,R<1-\vep,\; s\in(\vep, 2\vep)\right\}>0.
\]
Passing to the limit as $k\to \infty$ into \eqref{ps70} then due to \eqref{ps40} we deduce
\bean
\bar{u}(x,t)&&\geq N \omega_N\vep\;\tilde{C}(\vep)\lim_{k\to \infty}\int_{\vep}^{1-\vep} R^{N-1}  f_k(\tilde{u}_k(R,\widetilde{T}_\de))\,dR\\
&&\geq N \omega_N\vep\;\tilde{C}(\vep) \int_{\vep}^{1-\vep} R^{N-1}  v^p(R)\,dR=\infty,
\eean
which proves the assertion.
\end{proof}

\begin{corollary}\label{nik2}
Let $N\geq 3$ with  $\frac{N}{N-2}<p<p_S.$
Then the solution $u$ of \eqref{fkpp2a}-\eqref{id2} blows up completely.
\end{corollary}
\begin{proof}
The proof is an immediate consequence of Theorem \ref{cobu} and relation \eqref{nik}.
\end{proof}
\begin{rem}
Corollary \ref{nik2} actually means that the diffusion-driven instability stated by Theorem \ref{thmbu} is quite severe and thus any Turing (instability) pattern is destroyed once we exceed the blow-up time.
\end{rem}

\section{Blow-up rate and blow-up patterns}\label{blp}
Our aim in the current section is to determine the form the diffusion-driven blow-up ({\it DDBU}) provided  by Theorem \ref{thmbu}. We first provide some estimates of the blow-up rate for $u.$

\begin{proof}[Proof of Theorem \ref{tbu}]
We first observe that due to Lemma \ref{Lemma:mbetaestimate} and \eqref{mt2} there holds
\bea\label{nkl4}
D:=1-\sigma2^{\beta(\mu+1)/p} a_1^{\beta/p} \Lambda_1^{\beta\mu/p} \lambda^{\beta}<K(t)< 1<\infty,\quad\mbox{for any}\quad 0<t<T_{\de}.
\eea

Consider now $\Phi$ satisfying
\bean
&&\Phi_t=\Delta \Phi+\Phi^{p},\quad\mbox{in}\quad B_1\times\left(0,T_{\de}\right),\\
&&\frac{\partial \Phi}{\partial \nu}=0,\quad\mbox{on}\quad \partial B_1\times \left(0,T_{\de}\right), \\
&&\Phi(x,0)=u_0(x), \quad\mbox{in}\quad B_1,
\eean
then via comparison principle and due to \eqref{nkl4} we derive $u\leq \Phi$ in $ \bar{B}_1\times\left[0,T_{\de}\right].$

Yet it is known, see \cite[Theorem 44.6]{QS}, that
\bean
|\Phi(x,t)|\leq C_{\eta}|x|^{-\frac{2}{p-1}-\eta},\quad\mbox{in}\quad  B_1\times\left(0,T_{\de}\right)\quad\mbox{for some}\quad \eta>0,
\eean
and thus
\bea\label{tbsd18}
|u(x,t)|\leq C_{\eta}|x|^{-\frac{2}{p-1}-\eta}\quad\mbox{in}\quad  B_1\times\left(0,T_{\de}\right).
\eea
Then using standard parabolic estimates we get
\bea\label{nkl5}
u\in \mathcal{BUC^{\tau}}\left(\left\{\rho_0<|x|<1-\rho_0\right\}\times \left(\frac{T_{\de}}{2}, T_{\de}\right)\right),
\eea
for some $\tau\in(0,1)$ and each $0<\rho_0<1,$ where $\mathcal{BUC^{\tau}}(M)$ denotes the Banach space of all bounded and uniform $\tau-$H\"{o}lder continuous functions $\omega:M\subset\R^N\to \R,$ see also \cite{QS}.

Consequently \eqref{nkl5} infers that  $\lim_{t\to T_{\de}}u(x,t)$ exists and is finite for all $x\in B_1\setminus\{0\}.$
Recalling that $N>\displaystyle{\frac{2p}{p-1}}$ (or equivalently  $p>\displaystyle{\frac{N}{N-2}},\; N\geq 3$ ) then by using \eqref{nkl4},\eqref{tbsd18} and in view of dominated convergence theorem we derive
\bea\label{tbsd18a}
\lim_{t\to T_{\de}} K(t)=\gamma\in(0,+\infty).
\eea
Applying now Theorem 44.3(ii) in \cite{QS} and in conjunction with \eqref{tbsd18a} we can find a constant $C_{u}>0$ such that
\bea\label{ube}
\left|\left|u(\cdot,t)\right|\right|_{\infty}\leq C_{u}\left(T_{\de}-t\right)^{-\frac{1}{(p-1)}}\quad\mbox{in}\quad (0, T_{\de}).
\eea
On the other hand, setting $N(t):=\left|\left|u(\cdot,t)\right|\right|_{\infty}=u(0,t)$ then $N(t)$ is differentiable for almost every $t\in(0,T_{\de}),$ in view of  \cite{fmc85}, and it also satisfies
\bean
\frac{dN}{dt}\leq K(t) N^p(t).
\eean
Now since $K(t)\in C([0,T_{\de}))$ is bounded in any time interval $[0,t],\; t<T_{\de},$ and then upon integration over $(t,T_{\delta})$ we obtain
\bea\label{lbe}
\left|\left|u(\cdot,t)\right|\right|_{\infty}\geq C_l\left(T_{\de}-t\right)^{-\frac{1}{(p-1)}}\quad\mbox{in}\quad (0, T_{\de}),
\eea
for some positive constant $C_l$ and the proof is complete.
\end{proof}
\begin{rem}\label{nkl6}
Condition \eqref{ik2} implies that the diffusion-induced blow-up stated in Theorem \ref{thmbu} is of type I, i.e. the blow-up mechanism is controlled by the ODE part of \eqref{fkpp2a}.

\end{rem}
Next we  identify the blow-up  (Turing instability) pattern of the {\it DDBU solution} obtained  by Theorem \ref{thmbu}.

Note that \eqref{tbsd18} provides a rough form of the blow-up pattern for  $u.$  Nonetheless, due to \eqref{nkl4}  the non-local problem \eqref{fkpp2a}-\eqref{id2} can be tackled as the corresponding local one for which the following more accurate asymptotic blow-up profile, cf. \cite{mz98}, is available
\bea\label{kk1}
\lim_{t\to T_{max}}u(|x|,t)\sim C\left[\frac{|\log |x||}{|x|^2}\right]\quad\mbox{for}\quad |x|\ll 1.
\eea
For a more rigorous approach regarding non-local problems  the interested readers is advised to check \cite{DKZ20}. 

Relation \eqref{kk1}  provides the form of the blow-up profile of $u.$ Therefore \eqref{kk1}, in the biological context,  actually identifies the form of the developing patterns, which are induced as the result of the {\it DDI} phenomenon.
 



\end{document}